\documentclass[a4paper]{amsart}
\usepackage{graphicx}
\usepackage{amssymb}
\usepackage{amsmath}
\usepackage{amsthm}
\usepackage{amscd}
\usepackage{color}
\usepackage{colordvi}
\usepackage[all,2cell]{xy}
\usepackage{picture}
\usepackage{multirow}
\usepackage{makecell}
\usepackage{extarrows}

\usepackage{amsfonts,enumerate,verbatim,mathtools,tikz,bm,tikz-cd,hyperref,comment}
\hypersetup{
	colorlinks=true,
	linkcolor=blue,
	filecolor=magenta,
	urlcolor=cyan,
}

\UseAllTwocells \SilentMatrices
\newtheorem{thm}{Theorem}[section]

\newtheorem{cor}[thm]{Corollary}
\newtheorem{lem}[thm]{Lemma}
\newtheorem{exm}[thm]{Example}

\newtheorem{pro}[thm]{Proposition}
\theoremstyle{definition}
\newtheorem{dfn}[thm]{Definition}
\theoremstyle{remark}
\newtheorem{rem}[thm]{\bf Remark}
\numberwithin{equation}{section}

\DeclareMathOperator{\projdim}{\mathrm{proj.dim}}
\DeclareMathOperator{\injdim}{\mathrm{inj.dim}}
\DeclareMathOperator{\Hom}{\mathrm{Hom}}

\DeclareMathOperator{\Ext}{\mathrm{Ext}}
\DeclareMathOperator{\K}{\mathbb{K}}
\DeclareMathOperator{\Ker}{\mathrm{Ker}}

\DeclareMathOperator{\sgn}{\mathrm{sgn}}
\DeclareMathOperator{\proj}{\mathrm{proj}}

\begin{document}
\title{The preprojective algebra of a finite EI quiver}
\author{Dongdong Hu}
\subjclass[2020]{16G20, 16E30, 16D90, 18A05}
\date{\today}

\thanks{E-mail: hudd@mail.ustc.edu.cn}
\keywords{EI quiver, preprojective algebra, tensor algebra, Cartan matrix}

\maketitle

\dedicatory{}%
\commby{}%
%\begin{center}
%\end{center}

\begin{abstract}
	We define the preprojective algebra of a finite EI quiver. We prove that it is isomorphic to a centain tensor algebra. For a finite EI quiver of Cartan type, we prove that the corresponding preprojective algebra is isomorphic to the generalized preprojective algebra.
\end{abstract}
\maketitle
%\tableofcontents
%%%%%%%%%%%%%%%%%%%%
%%%%%%%%%%%%%%%%%%%%
%%%%%%%%%%%%%%%%%%%%
%----Section
\section{Introduction}
Let $\K$ be a field and $\mathcal{C}$ a finite EI category. Here, the EI assumption means that each endomorphism in $\mathcal{C}$ is an isomorphism. Denote by $\mathbb{K}\mathcal{C}$ its category algebra. We mention that category algebras play an important
role in the representation theory of finite groups; see \cite{We07,We08}. %The concept of a finite \emph{free} EI category is introduced in \cite{L2011}.

 Following \cite{L2011}, a \emph{finite EI quiver} $(Q,X)$ is a pair consisting of a finite  quiver $Q$ and an assignment $X=(X(i),X(\alpha))_{i\in Q_0,\alpha\in Q_1}$. More precisely, for each vertex $i\in Q_0$, $X(i)$ is a finite group, and for each arrow $\alpha\in Q_1$, $X(\alpha)$ is a finite non-empty $(X(t(\alpha)),X(s(\alpha)))$-biset. Each finite acyclic EI quiver gives rise a finite EI category $\mathcal{C}(Q,X)$, and a finite EI category is \emph{free} if it is isomorphic to $\mathcal{C}(Q,X)$ for some finite acyclic EI quiver $(Q,X)$.
 
  For a finite quiver $Q$, we have the preprojective algebra $\Pi(Q)$ \cite{GP,Rin}. We mention that the finite EI quiver generalizes the finite quiver.
So a natural question arises: how do we define the corresponding preprojective algebra of a finite EI quiver? 

Inspired by \cite{GLS, Geu2017}, we define the preprojective algebra $\Pi(Q,X)$ for a finite EI quiver $(Q,X)$; see Definition~\ref{Def2.2}. We prove that it is isomorphic to a centain tensor algebra in the following main result, which generalizes \cite[Theorem~A]{Rin}.
\vskip 5pt

\noindent {\bf Theorem\;A}\;(=\;Theorem~\ref{thm4.3}).\;\emph{Let $(Q, X)$ be a finite acyclic EI quiver with the assignment $X$ action-free. Set $\mathcal{C}=\mathcal{C}(Q,X)$ and $\mathbb{K}\mathcal{C}$ to be its the category algebra. Then the preprojective algebra $\Pi(Q,X)$ is isomorphic to the tensor algebra of the $\mathbb{K}\mathcal{C}$-$\mathbb{K}\mathcal{C}$-bimodule $\Ext_{\mathbb{K}\mathcal{C}}^1(D(\mathbb{K}\mathcal{C}),\mathbb{K}\mathcal{C})$.
}
\vskip 5pt

Here, an \emph{action-free} assignment $X$ means that the left $X(t(\alpha))$-action and right $X(s(\alpha))$-action on $X(\alpha)$ are both free for each $\alpha\in Q_1$. For a $\K$-vector space $V$, let $D(V)=\Hom_{\K}(V,\K)$ be the usual $\K$-dual of $V$.

Let $(C,D,\Omega)$ be a Cartan triple, that is, $C$ is a symmetrizable Cartan matrix, $D$ its symmetrizer and $\Omega$ an orientation of $C$. Two algebras $H(C,D,\Omega)$ and the generalized preprojective $\Pi(C,D,\Omega)$ are introduced in \cite{GLS}. The algebra $H(C,D,\Omega)$ generalizes the path algebra $\mathbb{K}Q$ and $\Pi(C,D,\Omega)$ generalizes the preprojective algebra $\Pi(Q)$. These provide connections between preprojective algebras, Lie algebras and cluster algebras from the symmetric case to the symmetrizable case.

 We mention the work of \cite{CW}, which associates a free EI category $\mathcal{C} = \mathcal{C}(C, D, \Omega)$ to each Cartan triple $(C, D, \Omega)$. Moreover, \cite{CW} establishes an algebra isomorphism between $\mathbb{K}\mathcal{C}$ and $H(C', D', \Omega')$ for another Cartan triple $(C',D',\Omega')$; see Section~\ref{sec6}. Based on \cite{CW}, we have the following result.

\vskip 5pt

\noindent {\bf Theorem\;B}\;(=\;Theorem~\ref{pro6.2}).\;\emph{Assume that $(C, D, \Omega)$ is a Cartan triple and that $\mathbb{K}$ has enough roots of unity for $D$. Let $(Q^\circ,X)$ be the EI quiver associated to the Cartan triple $(C,D,\Omega)$.  Then the preprojective algebra $\Pi(Q^\circ,X)$ is isomorphic to the generalized preprojective algebra $\Pi(C',D',\Omega').$
}
\vskip 5pt

The special case in Theorem~B is that the field $\mathbb{K}$ has characteristic $p$ and the entries of $D$ are $p$-power; see Corollary~\ref{cor6.3}. In this case, we have $(C',D',\Omega')=(C,D,\Omega)$. Then we have the preprojective algebra $\Pi(Q^\circ,X)$ is isomorphic to the generalized preprojective algebra $\Pi(C,D,\Omega).$

The paper is structured as follows. In Section \ref{sec2}, we recall some basic facts on free EI categories and define the preprojective algebra of finite EI quivers. We introduce the representation of finite EI quivers in Section \ref{sec3}. In Section~\ref{sec4}, we prove Proposition~\ref{pro4.2}, which is central to the proof of the main result. The main result Theorem \ref{thm4.3} is proved in Section~\ref{sec5}. In Section \ref{sec6}, we recall some algebras associated to Cartan matrices and prove $\Pi(Q^\circ,X)\cong\Pi(C',D',\Omega')$, where $(Q^\circ,X)$ is the EI quiver associated to the Cartan triple $(C,D,\Omega)$.

Given an algebra $A$ over a field $\mathbb{K}$, we write $A$-$\mathrm{mod}$ for the category of finite-dimensional left $A$-modules, and $A$-$\proj$ for the category of finite-dimensional projective left $A$-modules. For a left $A$-module $M$, we denote by $\projdim(M)$ (\emph{resp}. $\injdim(M)$) the projective dimension (\emph{resp.} injective dimension) of $M$. 
\section{Finite EI quivers and preprojective algebras}\label{sec2}
In this section, we introduce the preprojection algebra of a finite acyclic EI quiver; see Definition~\ref{Def2.2}. We prove that it is isomorphic to a tensor algebra of a certain bimodule over a category algebra; see Proposition~\ref{pro2.6}.

\subsection{The preprojective algebra}
Let $\K$ be a field and $\mathcal{C}$ be a category with finitely many objects. Denote by $\mathrm{Mor}(\mathcal{C})$ (\emph{resp}. $\mathrm{Obj}(\mathcal{C})$) the set of morphisms (\emph{resp}.  objects) in $\mathcal{C}$. The \emph{category algebra} $\mathbb{K}\mathcal{C}$ of $\mathcal{C}$ over $\K$ is defined as follows: $\mathbb{K}\mathcal{C}=\bigoplus_{\alpha\in\mathrm{Mor}(\mathcal{C})}\mathbb{K}\alpha$ as a $\K$-vector space and the product is induced by composition of morphisms:
\[\alpha\beta=\begin{cases}
	\alpha\circ\beta,& \mbox{if}\ \alpha\ \mbox{and}\ \beta\ \mbox{can be composed in}\ \mathcal{C};\\
	0,& \mbox{otherwise}.
\end{cases}\]
The unit of $\mathbb{K}\mathcal{C}$ is $1_{\mathbb{K}\mathcal{C}}=\sum\limits_{x\in\mathrm{Obj}(\mathcal{C})}e_x$, where $e_x$ denotes the identity endomorphism of $x$.

The category $\mathcal{C}$ is \emph{finite} if $\mathrm{Mor}(\mathcal{C})$ is a finite set. A finite category $\mathcal{C}$ is said to be EI \cite{L89} provided that every endomorphism is an isomorphism. Therefore, for each object $x$, $\mathrm{Hom}_\mathcal{C}(x,x)=\mathrm{Aut}_\mathcal{C}(x)$ is a finite group.

For a group $G$, denote by $\mathbb{K}G$ its group algebra. Let $G_1$ and $G_2$ be two groups. Recall that a $(G_1,G_2)$-\emph{biset} is a set $X$ with a left $G_1$-action and a right $G_2$-action such that $(g_1x)g_2=g_1(xg_2)$ for all $g_1\in G_1,x\in X,g_2\in G_2.$ For a $(G_1,G_2)$-biset $X$, we denote by $\mathbb{K}X$ the $\K$-vector space with a $\K$-basis $X$. Then $\mathbb{K}X$ is naturally a $\mathbb{K}G_1$-$\mathbb{K}G_2$-bimodule.

Let $X$ be a $(G_1,G_2)$-biset. Define its \emph{dual biset} $X^\ast=\{x^\ast|x\in X\}$ as follows:
\[x^\ast.g_1=(g_1^{-1}x)^\ast,\ \ g_2.x^\ast=(xg_2^{-1})^\ast,\mbox{ for } g_1\in G_1,g_2\in G_2,x\in X. \]
Then $X^\ast$ is a $(G_2,G_1)$-biset. We have an isomorphism of $(G_1,G_2)$-bisets
\begin{equation}\label{equ2.1}(X^\ast)^\ast\cong X\end{equation}
Then we can identify the bisets $(X^\ast)^\ast$ with $X$ via $(\ref{equ2.1})$. We have a natural isomorphism of $\mathbb{K}G_2$-$\mathbb{K}G_1$-bimodules
 \begin{equation}\label{iso1}\mathrm{Hom}_{\K}(\mathbb{K}X,\mathbb{K})\cong \mathbb{K}(X^\ast).\end{equation}
 Then we may identify the spaces $\mathrm{Hom}_{\K}(\mathbb{K}X,\mathbb{K})$ with $\mathbb{K}(X^\ast)$ via $(\ref{iso1})$.

Let $X$ be a $(G_1,G_2)$-biset and $Y$ be a $(G_2,G_3)$-biset. The \emph{biset product} of $X$ and $Y$, denoted by $X\times_{G_2} Y$, is the set $X\times Y/\sim$ with the equivalence relation $(x,gy)\sim(xg,y)$ for $x\in X,g\in G_2,y\in Y.$ For brevity, the equivalence class of $(x,y)$ for $x\in X,y\in Y$ is still denoted by $(x,y)$. The set $X\times_{G_2}Y$ is naturally a $(G_1,G_3)$-biset.

Recall that a finite quiver $Q=(Q_0,Q_1,s,t)$ is a quadruple consisting of a finite set $Q_0$ of vertices, a finite set $Q_1$ of arrows, and two maps $s,t:Q_1\rightarrow Q_0$. We write $s(\alpha)$ and $t(\alpha)$ for the source and target of an arrow $\alpha\in Q_1$. A path $p=\alpha_n\cdots \alpha_2\alpha_1$ of length $n\ge 2$ consists of arrows $\alpha_i$ such that $t(\alpha_i)=s(\alpha_{i+1})$ for all $1\le i< n$. Here, we write the concatenation from the right to the left, and set $s(p)=s(\alpha_1),\ t(p)=t(\alpha_n)$. The paths of lengths zero and one coincide with the elements of $Q_0$ and $Q_1$, respectively. The trival path, denoted by $e_i$, is the path of length zero whose source and target are the vertex $i\in Q_0$. A quiver is \emph{acyclic}, if it does not contain a non-trivial path $p$ such that $s(p)=t(p)$. Let $Q_n$ be the set of paths of length $n$, and denote by $\mathbb{K}Q=\bigoplus_{n\ge 0}\mathbb{K}Q_n$ the path algebra, whose multiplication is induced by
the concatenation of paths.

Recall from \cite[Definition 2.1]{L2011} that a \emph{finite EI quiver} $(Q,X)$ is a pair consisting of a finite  quiver $Q$ and an assignment $X=(X(i),X(\alpha))_{i\in Q_0,\alpha\in Q_1}$. More precisely, for each vertex $i\in Q_0$, $X(i)$ is a finite group, and for each arrow $\alpha\in Q_1$, $X(\alpha)$ is a finite non-empty $(X(t(\alpha)),X(s(\alpha)))$-biset. A finite EI quiver $(Q,X)$ is called \emph{acyclic}, if the quiver $Q$ is acyclic. For a path $p=\alpha_n\cdots\alpha_2\alpha_1$ of length $n$, we define 
\[X(p)=X(\alpha_n)\times_{X(t(\alpha_{n-1}))}X(\alpha_{n-1})\times_{X(t(\alpha_{n-2}))}\cdots\times_{X(t(\alpha_2))}X(\alpha_2)\times_{X(t(\alpha_1))}X(\alpha_1).\]
Then $X(p)$ is naturally a $(X(t(p)),X(s(p)))$-biset. Indeed, for two paths $p,q$ satisfying $s(p)=t(q)$, we have a natural isomorphism
\begin{equation}\label{iso}\begin{tikzcd}
		X(p)\times_{X(t(q))}X(q) \arrow[r, "\sim"] & X(pq)
\end{tikzcd}\end{equation}
where $pq$ denotes the concatenation. We identify $X(e_i)$ with $X(i)$.

Each finite EI quiver $(Q,X)$ determines a category $\mathcal{C}=\mathcal{C}(Q,X)$; see \cite[Section~2]{L2011}. The objects of $\mathcal{C}$ are the vertices of $Q$; for two objects $x$ and $y$, we have the disjoint union
\[\mathrm{Hom}_\mathcal{C}(x,y)=\bigsqcup\limits_{\{p\ \mbox{path in}\ Q\ \mbox{with}\ s(p)=x\ \mbox{and}\ t(p)=y\}}X(p).\]
The composition of morphisms in $\mathcal{C}$ is induced by the concatenation of paths and the isomorphism (\ref{iso}). We observe that $\mathcal{C}$ is a finite EI category if and only if the quiver $Q$ is acyclic. A finite EI category $\mathcal{C}$ is called \emph{free}, if it is isomorphic to $\mathcal{C}(Q,X)$ for some finite acyclic EI quiver $(Q,X)$; see \cite[Definition~2.2]{L2011}.

Let $A$ be a $\K$-algebra, and let $V$ be an $A$-$A$-bimodule, where $\K$ acts on $V$ centrally. Denote by $$T_A(V)=\K\oplus V\oplus V^{\otimes_A 2}\oplus\cdots \oplus V^{\otimes_A n}\cdots $$ the thesor algebra. The following result is proved in \cite{CW} for the acyclic case, whose proof clearly applies to the general case.

\begin{pro}\label{CWpro1.1}$($\cite[Proposition~2.2]{CW}$)$
	Let $\mathcal{C}=\mathcal{C}(Q,X)$ be the category associated to a finite EI quiver $(Q,X)$. Set $A=\prod_{x\in Q_0}\mathbb{K}X(x)$ and $V=\bigoplus_{\alpha\in Q_1}\mathbb{K}X(\alpha)$, where $V$ is naturally an $A$-$A$-bimodule. Then there is an algebra isomorphism $T_A(V)\rightarrow\mathbb{K}\mathcal{C}$ such that its restriction to $A\oplus V$ is the identity.
\end{pro}

For a finite acyclic EI quiver $(Q,X)$, we define a new finite EI quiver $(\overline{Q},\overline{X})$ as follows: the double quiver $\overline{Q}$ is obtained by adjoining a reverse arrow $\alpha^\ast:j\rightarrow i$ for each arrow $\alpha:i\rightarrow j$ in $Q$, and the assignment $\overline{X}$ is given by
\[\overline{X}(i)=X(i),\ \forall i\in Q_0,\ \overline{X}(\alpha)=X(\alpha),\ \overline{X}(\alpha^\ast)=X(\alpha)^\ast,\ \forall\alpha\in Q_1.\]
 For an arrow $\alpha\in Q_1$, we recall that $X(\alpha)$ is a $(X(t(\alpha)),X(s(\alpha)))$-biset. Denote by $L_\alpha$ a representative set of  $X(t(\alpha))$-orbits of $X(\alpha)$, and by $R_\alpha$ a representative set of  $X(s(\alpha))$-orbits of $X(\alpha)$.
 
  The following definition is inspired by \cite[Definition~4.5]{K} and \cite[Definition~3.3.4]{Geu2017}.

\begin{dfn}\label{Def2.2}
	Let $(Q,X)$ be a finite acyclic EI quiver. We define the \emph{preprojective algebra} of $(Q,X)$ by
	\[\Pi(Q,X):=\mathbb{K}\mathcal{C}(\overline{Q},\overline{X})/(\rho),\]
	where $\rho=\sum_{\alpha\in Q_1}\left(\sum_{b\in R_\alpha}bb^\ast-\sum_{c\in L_\alpha}c^\ast c\right)$.
\end{dfn}
\begin{rem}\label{rem2.3}
	\begin{enumerate}[(1)]
		\item  We observe that $(bg)(bg)^\ast=(bg)(g^{-1}b^\ast)=bb^\ast$ for all $b\in R_\alpha$ and $g\in X(s(\alpha))$. Similarly, we have $(gc)^\ast(gc)=c^\ast c$ for all $c\in L_\alpha$ and $g\in X(t(\alpha)).$ Therefore the element $\rho\in \mathbb{K}\mathcal{C}(\overline{Q},\overline{X})$ in definition of $\Pi(Q,X)$ does not depend on the choices of $R_\alpha$ and $L_\alpha$.
		\item For each $i\in Q_0$, let $$\rho_i=\sum_{\alpha\in Q_1\atop t(\alpha)=i}\sum_{b\in R_\alpha}bb^\ast-\sum_{\beta\in Q_1\atop s(\beta)=i}\sum_{c\in L_\beta}c^\ast c.$$ We observe $(\rho)=(\rho_i,i\in Q_0)$.
		\item Assume that the above assignment $X$ is trivial, that is, all set $X(i)$ and $X(\alpha)$ are consisting of single elements. Then the preprojective algebra $\Pi(Q,X)$ is the classical preprojective algebra $\Pi(Q)$ \cite{Rin} associated to $Q$.
	\end{enumerate}
\end{rem}
\begin{comment}
Let $(Q,X)$ be a finite EI quiver. We call $X$ \emph{action-free}, if the left $X(t(\alpha))$-action and right $X(s(\alpha))$-action on $X(\alpha)$ are both free for each $\alpha\in Q_1$. 

The following result, proved in Section \ref{sec4}, is an analogue of \cite[Theorem A]{Rin} for classical preprojective algebra $\Pi(Q)$.
\begin{thm}
	Let $\mathcal{C}=\mathcal{C}(Q,X)$ be the category associated to a finite acyclic EI quiver $(Q,X)$, and $\Lambda=\K\mathcal{C}$. If $X$ is action-free, then we have an algebra isomorphism $\Pi(Q,X)\cong T_\Lambda(\Ext_\Lambda^1(D(\Lambda),\Lambda))$.
\end{thm}
\end{comment}
\subsection{The algebra $\Pi$ as a tensor algebra}

Let $S$ be a $\K$-algebra, and $V_0,V_1$ be two $S$-$S$-bimodules. We equip the tensor algebra $\Lambda=T_S(V_0\oplus V_1)$ with an $\mathbb{N}$-grading by setting $|s|=0$ for $s\in S$, $|x|=0$ for $x\in V_0$ and $|y|=1$ for $y\in V_1$. Let $\rho\in \Lambda$ with degree one, and define $\Pi=\Lambda/(\rho)$. Then $\Pi$ is also $\mathbb{N}$-graded. Denote by $\Lambda_i$  (\emph{resp}. $\Pi_i$) the subspace of  $\Lambda$ (\emph{resp}. $\Pi$) consisting of elements with degree $i$. 

For the universal property of the tensor algebra, we refer to \cite[Section~1]{BSZ}. The following lemma shows that $\Pi$ is a tensor algebra; compare \cite[Proposition~6.3]{GLS}.

\begin{lem}\label{lem2.5}
  	Keep the notation as above. Then we have the following statements.
  	
  	\begin{enumerate}[(1)]
  		\item $\Lambda_0=T_S(V_0)$ and $\Lambda_1$ is the sub $\Lambda_0$-$\Lambda_0$-bimodule of $\Lambda$ generated by $V_1$. 
  		\item $\Pi_0\cong\Lambda_0$, and $\Pi_1$ is the sub $\Pi_0$-$\Pi_0$-bimodule of $\Pi$ generated by the image of $V_1$ in $\Pi$. 
  		\item $\Pi_1\cong (\Lambda_0\otimes_S V_1\otimes \Lambda_0)/\Lambda_0\rho\Lambda_0$.
  		\item $\Pi\cong T_{\Pi_0}(\Pi_1)$.
  	\end{enumerate}
\end{lem}
\begin{proof}
	(1) We observe that 
	\begin{equation}\label{equ2.4}
	\Lambda_0=\bigoplus_{i\ge 0}V_0^{\otimes_S i}
	\end{equation}
	 and
	\begin{equation}\label{equ2.5}
	 \Lambda_1=\bigoplus_{i,j\ge 0}V_0^{\otimes_S i}\otimes_S V_1\otimes_S V_0^{\otimes_S j}=\Lambda_0\otimes_S V_1\otimes_S\Lambda_0. 
	 \end{equation}
	
	(2) It follows from the facts that $\Pi_i\cong\Lambda_i/(\Lambda_i\cap (\rho))$ and $\Lambda_0\cap (\rho)=0$.
	
	(3) We have 
	\[(\rho)=\Lambda\rho\Lambda=\bigoplus\limits_{i,j\ge 0}\Lambda_i\rho\Lambda_j.\]
	Since $\rho\in\Lambda_1$, we have $\Lambda_1\cap(\rho)=\Lambda_0\rho\Lambda_0$. Therefore, we have $$\Pi_1\cong\Lambda_1/(\Lambda_1\cap(\rho))=(\Lambda_0\otimes_S V_1\otimes_S\Lambda_0)/\Lambda_0\rho\Lambda_0.$$
	
	(4) Let
	\[	
	f_0:S \hookrightarrow  \Lambda_0  \twoheadrightarrow  \Pi_0\hookrightarrow T_{\Pi_0}(\Pi_1),\;\mbox{and} 
	\]\[
	f_1:V_0\oplus V_1 \hookrightarrow  \Lambda_0\oplus \Lambda_1\twoheadrightarrow  \Pi_0\oplus \Pi_1 \hookrightarrow T_{\Pi_0}(\Pi_1).
	\]
	Since $f_0$ is an algebra homomorphism and $f_1$ is an $S$-$S$-bimodule homomorphism, by the universal property of tensor algebras, there is a unqiue homomorphism of algebras $$f:\Lambda\longrightarrow T_{\Pi_0}(\Pi_1)$$ such that $f|_S=f_0$ and $f|_{V_0\oplus V_1}=f_1$.  By the fact (\ref{equ2.5}), we observe that
	\begin{equation}\label{equ2.6}
		f(x)=x+(\Lambda_1\cap (\rho)):=\overline{x}\in\Pi_1,
		\end{equation}
	 for any $x\in \Lambda_1$. 
	Thus we have $f(\rho)=0$ and a homomorphism of algebras induced by $f$:
	\[\widetilde{f}:\Pi=\Lambda/(\rho)\longrightarrow T_{\Pi_0}(\Pi_1).\]
	
	On the other hand, there is a unique homomorphsim of algebras 
	$$g:T_{\Pi_0}(\Pi_1)\longrightarrow \Pi$$
	 such that $g|_{\Pi_0}=\mathrm{Id}_{\Pi_0}$ and $g|_{\Pi_1}=\mathrm{Id}_{\Pi_1}$. 
	 %Since $\widetilde{f}g(\overline{x})=\widetilde{f}(\overline{x})=f(x)=\overline{x}$, for any $x\in \Lambda_0$ or $x\in\Lambda_1$, 
	 By (\ref{equ2.4}), we have $$f(y)=y+(\Lambda_0\cap (\rho)):=\overline{y}\in\Pi_0,$$ for any $y\in \Lambda_0$. We obtain $$\widetilde{f}g(\overline{y})=\widetilde{f}(\overline{y})=f(y)=\overline{y},$$
	  for any $y\in \Lambda_0$. 
	 % we have $x\in S$ or $x=x_1\otimes x_2\otimes\cdots\otimes x_r$ for some $x_1,x_2,\cdots,x_r\in M_0$. If $x\in S$, then we have $f(x)=f_0(x)=\overline{x}$, if $x=x_1\otimes x_2\otimes\cdots\otimes x_r$, then we have
	 %\[\begin{aligned}
	 %	f(x)&=f(x_1\otimes x_2\otimes\cdots\otimes x_r)\\
	 %	&=f(x_1)\otimes f(x_2)\otimes\cdots\otimes f(x_r)\\
	 %	&=f_1(x_1)\otimes f_1(x_2)\otimes\cdots\otimes f_1(x_r)\\
	 %	&=\overline{x_1}\otimes\overline{x_2}\otimes\cdots\otimes\overline{x_r}\\
	 %	&=\overline{x}.
	 %\end{aligned}\]
	Similarly, using (\ref{equ2.6}), we have $\widetilde{f}g(\overline{x})=\overline{x}$ for $x\in\Lambda_1$.
	Therefore we have $\widetilde{f}\circ g=\mathrm{Id}_{T_{\Pi_0}(\Pi_1)}$. Let $\pi:\Lambda\rightarrow \Lambda/(\rho)=\Pi$ be the canonical homomorphism, so we have $\widetilde{f}\circ\pi=f.$ We have
	\[gf(s)=gf_0(s)=g(\overline{s})=\overline{s}=\pi(s),\;\mbox{for all } s\in S.\]
	Here, we identify $\overline{s}=s+(\Lambda_0\cap (\rho))$ with $s+(\rho)=\pi(s)$ in $\Pi$. Similarly, we have
	\[gf(m)=g(\overline{m})=\overline{m}=\pi(m),\;\mbox{for all } m\in V_0\oplus V_1.\] 
	By the universal property of tensor algebras, we immediately obtain $g\circ f=\pi$. 
	So we have $$g\circ\widetilde{f}\circ\pi=g\circ f=\pi,\mbox{ and } g\circ\widetilde{f}=\mathrm{Id}_\Pi.$$ Therefore $\widetilde{f}:\Pi\rightarrow T_{\Pi_0}(\Pi_1)$ is an isomorphism.
\end{proof}

For a finite acyclic EI quiver $(Q,X)$, let $\mathcal{C}=\mathcal{C}(Q,X)$. We set 
$$
A=\prod_{i\in Q_0} \mathbb{K}X(i),\; V_0=\bigoplus_{\alpha\in Q_1}\mathbb{K}X(\alpha),\mbox{ and } V_1=\bigoplus_{\alpha\in Q_1}\mathbb{K}X(\alpha^\ast).
$$
Then $V_0$ and $V_1$ are naturally $A$-$A$-bimodules. Applying Proposition~\ref{CWpro1.1} to $(\overline{Q},\overline{X})$, we have algebra isomorphisms $\mathbb{K}\mathcal{C}(\overline{Q},\overline{X})\cong T_A(V_0\oplus V_1):=\Lambda$ and
\begin{equation}\label{equ2.7}\Pi(Q,X)\cong T_A(V_0\oplus V_1)/(\rho):=\Pi,\end{equation}
 where $\rho=\sum_{\alpha\in Q_1}\left(\sum_{b\in R_\alpha}b\otimes b^\ast-\sum_{c\in L_\alpha}c^\ast \otimes c\right)$. Therefore, by the setting in Lemma~\ref{lem2.5}, we obtain that $\Pi(Q,X)$ is $\mathbb{N}$-graded, where the grading is given as follows: for each $i\in Q_0$ and $\alpha\in Q_1$, we set $|a|=0$ for $a\in X(i)$, $|x|=0$ for $x\in X(\alpha)$ and $|y|=1$ for $y\in X(\alpha^\ast)$. 
\begin{lem}\label{lem2.6}
	Keep the notation as above.	The composite homomorphism $
	\mathbb{K}\mathcal{C}\hookrightarrow\mathbb{K}\mathcal{C}(\overline{Q},\overline{X})\twoheadrightarrow\Pi(Q,X)
	$ induces an isomorphism $\mathbb{K}\mathcal{C}\xrightarrow{\sim}\Pi(Q,X)_0$.
\end{lem}
\begin{proof} We have the following commutative diagram:
$$	\begin{tikzcd}
		T_A(V_0) \arrow[r, hook] \arrow[d, "\cong"]    & \Lambda \arrow[r, two heads] \arrow[d,"\cong"]            & \Pi \arrow[d,"\cong"] \\
		\mathbb{K}\mathcal{C} \arrow[r, hook] & {\mathbb{K}\mathcal{C}(\overline{Q},\overline{X})} \arrow[r, two heads] & {\Pi(Q,X)}   
	\end{tikzcd}
$$
where the vertical isomorphisms on the left side and in the middle are given by Proposition~\ref{CWpro1.1}. By Lemma~\ref{lem2.5} (1) and (2), we have $T_{A}(V_0)=\Lambda_0\cong\Pi_0$, hence we have an algebra isomorphism $\mathbb{K}\mathcal{C}=\mathbb{K}\mathcal{C}(\overline{Q},\overline{X})_0\xrightarrow{\sim}\Pi(Q,X)_0$.
\end{proof}

 We can identify $\mathbb{K}\mathcal{C}$ with $\Pi(Q,X)_0$ via the algebra isomorphism in Lemma~\ref{lem2.6}. Then $\Pi(Q,X)_1$ is a $\mathbb{K}\mathcal{C}$-$\mathbb{K}\mathcal{C}$-bimodule. By Lemma~\ref{lem2.5} (2), we have isomorphism of algebras 
 \begin{equation}\label{equ2.9}
 	\Pi_0\cong \Lambda_0= T_A(V_0)\cong\mathbb{K}\mathcal{C}.
 \end{equation}
  Then we have that $\Pi_1$ is also a $\mathbb{K}\mathcal{C}$-$\mathbb{K}\mathcal{C}$-bimodule. By the isomorphism (\ref{equ2.7}), we have an isomorphism of $\mathbb{K}\mathcal{C}$-$\mathbb{K}\mathcal{C}$-bimodules \begin{equation}\label{equ2.8}\Pi_1\cong\Pi(Q,X)_1.\end{equation} 
 Then $\Pi(Q,X)_1$ is precisely the sub $\mathbb{K}\mathcal{C}$-$\mathbb{K}\mathcal{C}$-bimodule of $\Pi(Q,X)$ generated by $X(\alpha^\ast)$ for all $\alpha\in Q_1$.
\begin{pro}\label{pro2.6}
	Keep the notation as above. Then we have the following statements.
	\begin{enumerate}[(1)]
		\item $\Pi(Q,X)_1\cong(\mathbb{K}\mathcal{C}\otimes_A V_1 \otimes_A\mathbb{K}\mathcal{C})/\mathbb{K}\mathcal{C}\rho'\mathbb{K}\mathcal{C}$. Here, $$\rho'=\sum_{\alpha\in Q_1}\left(\sum_{b\in R_\alpha}b\otimes b^\ast\otimes1-\sum_{c\in L_\alpha}1\otimes c^\ast\otimes c\right).$$
		\item $\Pi(Q,X)\cong T_{\mathbb{K}\mathcal{C}}(\Pi(Q,X)_1)$.
	\end{enumerate}
\end{pro}
\begin{proof}
	(1) By (\ref{equ2.9}), (\ref{equ2.8}) and Lemma~\ref{lem2.5}~(3), we have 
	\[\Pi(Q,X)_1\cong\Pi_1\cong(\Lambda_0\otimes_A\otimes V_1\otimes_A\Lambda_0)/\Lambda_0\rho\Lambda_0\cong(\mathbb{K}\mathcal{C}\otimes_A V_1 \otimes_A\mathbb{K}\mathcal{C})/\mathbb{K}\mathcal{C}\rho'\mathbb{K}\mathcal{C}.\]
	
(2)	We have the following isomorphisms of algebras
	\[\Pi(Q,X)\cong\Pi\cong T_{\Pi_0}(\Pi_1)\cong T_{\mathbb{K}\mathcal{C}}(\Pi(Q,X)_1).\] 
	Starting from the left side, the first isomorphism is (\ref{equ2.7}), the second one is given by Lemma~\ref{lem2.5} (3) and the last one follows from the isomorphisms $\Pi_0\cong\Lambda_0\cong\mathbb{K}\mathcal{C}$ and $(\ref{equ2.8})$.
\end{proof}

\section{Representations of a finite EI quiver}\label{sec3}

We fix a finite acyclic EI quiver $(Q,X)$, where $X$ is \emph{action-free}, that is, the left $X(t(\alpha))$-action and right $X(s(\alpha))$-action on $X(\alpha)$ are both free for each $\alpha\in Q_1$.  Then $\mathbb{K}X(\alpha)$ is free as a right $\mathbb{K}X(s(\alpha))$-module and free as a left $\mathbb{K}X(t(\alpha))$-module. Denote by $A_i=\mathbb{K}X(i)$ the group algebra of $X(i)$ for each $i\in Q_0$.

In this section, we make preparation on the proof of Theorem~\ref{pro4.2}. The aim of this section is to obtain the commutative diagram; see Proposition~\ref{lem3.10}. 

\subsection{Locally projective representations}

In this subsection, we introduce the representation of a finite EI quiver. In fact, we are interested in the locally projective representation \cite{K}. By applying \cite[Proposition~3.5]{K}, we obtain the homological characterization of locally projective repersentations; see Proposition~\ref{pro3.3}.

The representation theory of a finite EI quiver is similar to the representation theory of a modulated graph \cite{DR}; compare \cite[Section~5]{GLS} and \cite[Section~2]{K}. A \emph{representation} $M=(M_i,M_\alpha)_{i\in Q_0,\alpha\in Q_1}$ of $(Q,X)$ consists of a finite-dimensional $A_i$-module $M_i$ for each $i\in Q_0$ and a homomorphism of $A_i$-modules
\[M_\alpha:\mathbb{K}X(\alpha)\otimes_{A_j} M_j\rightarrow M_i,\]
for each arrow $\alpha:j\rightarrow i\in Q_1$. A \emph{morphism} $f:M\rightarrow N$ of representations of $(Q,X)$ is a tuple $f=(f_i)_{i\in Q_0}$ of homomorphisms of $A_i$-modules $f_i:M_i\rightarrow N_i$ such that the diagram
\[\begin{tikzcd}
	\mathbb{K}X(\alpha)\otimes_{A_j}M_j \arrow[r, "M_\alpha"] \arrow[d,"1\otimes_{A_j} f_j"'] & M_i \arrow[d, "f_i"] \\
	\mathbb{K}X(\alpha)\otimes_{A_j}N_j \arrow[r, "N_\alpha"]                                 & N_i                 
\end{tikzcd}\]
commutes for each $\alpha:j\rightarrow i\in Q_1$. The representations of $(Q,X)$ form a category $\mathrm{rep}(Q,X)$. $A$ representation $M\in\mathrm{rep}(Q,X)$ is called \emph{locally projective}, if each $A_i$-module $M_i$ is projective. Let $\mathrm{rep}_{l.p.}(Q,X)$ be the subcategory of $\mathrm{rep}(Q,X)$ consisting of all locally projective representations.

 The following proposition generalises \cite[Proposition~III.1.7]{ARS}; compare \cite[Section~8]{M} and \cite[Proposition~1.1]{Xu}.
\begin{pro}\label{pro3.1}
	The categories $\mathrm{rep}(Q,X)$ and $\mathbb{K}\mathcal{C}(Q,X)$-$\mathrm{mod}$ are equivalent.
\end{pro}
\begin{proof}
	We write $\mathcal{C}=\mathcal{C}(Q,X)$. For a representation $M=(M_i,M_\alpha)_{i\in Q_0,\alpha\in Q_1}$ of $(Q,X)$, we define a $\mathbb{K}\mathcal{C}$-module $F(M)$ as follows: $F(M)=\bigoplus_{i\in Q_0}M_i$ as a $\K$-vector space. To define a $\mathbb{K}\mathcal{C}$-module structure on $F(M)$, it suffices to define the products of the form $xm_k$, where $x\in X(p)$ for $p \in Q_0\sqcup Q_1$, and $m_k\in M_k$. If $p=i\in Q_0$, we set $xm_k=xm_i$ in the $A_i$-module $M_i$ for $k=i$, and set $xm_k=0$ for $k\ne i$. If $p=\alpha\in Q_1$ from $i$ to $j$, we set $xm_k=M_\alpha(x\otimes m_k)$ for $k=i$, and set $xm_k=0$ for $k\ne i$. 
	Then $F(M)$ is a $\mathbb{K}\mathcal{C}$-module. If $f=(f_i)_{i\in Q_0}:M\rightarrow N$ is a morphism in $\mathrm{rep}(Q,X)$, then $F(f)=\bigoplus_{i\in Q_0}f_i:F(M)\rightarrow F(N)$ is a $\mathbb{K}\mathcal{C}$-module homomorphism. Then we get a functor $F:\mathrm{rep}(Q,X)\rightarrow \mathbb{K}\mathcal{C}$-$\mathrm{mod}$.
	
	Conversely, let $Y$ be a $\mathbb{K}\mathcal{C}$-module. We define a representation $G(Y)=(Y_i,Y_\alpha)$ as follows: for each $i\in Q_0$, $Y_i=e_i Y$ which is an $A_i$-module, and for each $\alpha\in Q_1$,  $Y_\alpha:\mathbb{K}X(\alpha)\otimes_{A_{s(\alpha)}} Y_{s(\alpha)}\rightarrow Y_{t(\alpha)}$ is given by $v\otimes m\mapsto vm$, which is an $A_{t(\alpha)}$-module homorphism. Then $G(Y)$ is a representation of $(Q,X)$. If $g:Y\rightarrow Y'$ is a $\mathbb{K}\mathcal{C}$-module homomorphism, then $G(g)=(g_i)_{i\in Q_0}:G(Y)\rightarrow G(Y')$ is a morphism in $\mathrm{rep}(Q,X)$, where $g_i=g|_{e_iY}$ for each $i\in Q_0$. Then we get a functor $G:\mathbb{K}\mathcal{C}$-\rm{mod}$\rightarrow\mathrm{rep}(Q,X)$.
	
	We observe that the functors $F$ and $G$ are quasi-inverse to each other. Thus the categories $\mathrm{rep}(Q,X)$ and $\mathbb{K}\mathcal{C}$-$\mathrm{mod}$ are equivalent.
\end{proof}
%\begin{rem}
%	We identify the category $\mathrm{rep}(Q,X)$ with $\K\mathcal{C}(Q,X)$-$\rm{mod}$, and for $X$ is trivial,	that is, all sets $X(i)$ and $X(\alpha)$ are consisting of single elements, the equivalence in Proposition \ref{pro3.1} in this case is well known; see \cite[Proposition~III.1.7]{ARS}.
%\end{rem}

The following result is the homological characterization of locally projective representations; compare \cite[Theorem~1.2]{GLS}.
\begin{pro}\label{pro3.3}
	Let $(Q,X)$ be a finite acyclic EI quiver with the assignment $X$ action-free.	For $M\in\mathrm{rep}(Q,X)$, the following statements are equivalent:
	\begin{enumerate}[(1)]
		\item $\projdim(M)\le 1$;
		\item $\injdim(M)\le 1$;
	%	\item $\projdim(M)<\infty$;
	%	\item $\injdim(M)<\infty$;
		\item $M\in\mathrm{rep}_{l.p.}(Q,X)$.
	\end{enumerate}
\end{pro}
\begin{proof}
	Since the group algebra $\mathbb{K}X(i)$ is a selfinjective algebra for each $i\in Q_0$, the result follows from the case $n=0$ of \cite[Proposition~3.5]{K}.
\end{proof}
By Proposition~\ref{pro3.1}, we can identify the categories $\mathrm{rep}(Q,X)$ with $\mathbb{K}\mathcal{C}(Q,X)$-$\rm{mod}$. The following result is analogous to \cite[Corollary~7.2]{GLS}.

\begin{lem}\label{lem3.2}
		Let $(Q,X)$ be a finite acyclic EI quiver with $X$ action-free and let $\mathcal{C}=\mathcal{C}(Q,X)$. For $M\in\mathrm{rep}_{l.p.}(Q,X)$,	we have a projective resolution of $M$:
\begin{equation}\label{3.1}
			0 \to \bigoplus\limits_{\alpha\in Q_1}\mathbb{K}\mathcal{C} e_{t(\alpha)}\otimes_{A_{t(\alpha)}}\mathbb{K}X(\alpha)\otimes_{A_{s(\alpha)}}M_{s(\alpha)} \xrightarrow{d} \bigoplus\limits_{i\in Q_0}\mathbb{K}\mathcal{C} e_i\otimes_{A_i}M_i \xrightarrow{\mu} M \to 0
	\end{equation}
	where
	\[d(p\otimes v\otimes m)=p\otimes M_\alpha(v\otimes m)-pv\otimes m,\mbox{ for }p\in \mathbb{K}\mathcal{C} e_{s(\alpha)},\; v\in\mathbb{K}X(\alpha),\mbox{ and } m\in M_{s(\alpha)};\]
	\[\mu(p\otimes m)=pm,\mbox{ for }p\in \mathbb{K}\mathcal{C} e_i\mbox{ and }\ m\in M_i.\]
\end{lem}

\begin{proof}
	Let $A=\prod_{i\in Q_0}A_i$ and $V=\bigoplus_{\alpha\in Q_1}\mathbb{K}X(\alpha)$. By Proposition~\ref{CWpro1.1}, we have an algebra isomorphism $\mathbb{K}\mathcal{C}\cong T_A(V)$. By \cite[Theorems~10.1 and~10.5]{S}, we have a short exact sequence of $\mathbb{K}\mathcal{C}$-$\mathbb{K}\mathcal{C}$-bimodules
	\begin{equation}\label{equ3.2.2}
		0 \rightarrow \mathbb{K}\mathcal{C}\otimes_A V\otimes_A \mathbb{K}\mathcal{C} \xrightarrow{d'} \mathbb{K}\mathcal{C}\otimes_A \mathbb{K}\mathcal{C} \xrightarrow{\mathrm{mult}} \mathbb{K}\mathcal{C} \rightarrow 0.
		\end{equation}
	Here, $d'(p\otimes v\otimes q)=p\otimes vq-pv\otimes q$; see also \cite{Rog}. It splits as a short exact sequence of right  $\mathbb{K}\mathcal{C}$-modules. Applying $-\otimes_{\mathbb{K}\mathcal{C}} M$ to the sequence (\ref{equ3.2.2}) yields a short exact sequence of left $\mathbb{K}\mathcal{C}$-modules
	\[
		0 \rightarrow \mathbb{K}\mathcal{C}\otimes_A V\otimes_A M \xrightarrow{d} \mathbb{K}\mathcal{C}\otimes_A M \xrightarrow{\mu}  M \rightarrow  0.
	\]
	Here, $d(p\otimes v\otimes m)=p\otimes vm-pv\otimes m$ and $\mu(p\otimes m)=pm$. It is isomorphic to the sequence (\ref{3.1}). 
	
	Since $\mathbb{K}X(\alpha)$ is a free left $A_{t(\alpha)}$-module, we have that $\mathbb{K}\mathcal{C} e_{t(\alpha)}\otimes_{A_{t(\alpha)}}\mathbb{K}X(\alpha)$ is a projective left $\mathbb{K}\mathcal{C}$-module. We have that $\mathbb{K}\mathcal{C} e_{t(\alpha)}\otimes_{A_{t(\alpha)}}\mathbb{K}X(\alpha)\otimes_{A_{s(\alpha)}}M_{s(\alpha)}$ is a projective left $\mathbb{K}\mathcal{C}$-module since $M$ is locally projective. Similarly, for each $i\in Q_0$, $\mathbb{K}\mathcal{C} e_i\otimes_{A_i}M_i$ is also a projective left $\mathbb{K}\mathcal{C}$-module. 
\end{proof}
\subsection{Trace maps}\label{subsec3.2}
In this subsection, we introduce the trace map \cite{Bou,Geu2017}.  

Recall that for a $\K$-algebra $R$ and a finitely generated projective left $R$-module $P$, there exist $x_1,x_2,\cdots,x_m\in P$ and $\theta_1,\theta_2,\cdots,\theta_m\in\Hom_R(P,R)$ such that for each $x\in P,\ x=\sum_{i=1}^m\theta_i(x)x_i$. The elements $x_1,x_2,\cdots,x_m\in P$ and $\theta_1,\theta_2,\cdots,\theta_m\in\mathrm{Hom}_R(P,R)$ are called a \emph{dual basis} of $P$.

Let $R$ be a symmetric $\K$-algebra with a symmetric structure $\varphi:R\to \K$. Here, a symmetric structure \cite[Theorem~16.54]{Lam} means that $\varphi$ is a $\K$-linear map that satisfies the following properties:
\begin{enumerate}[(1)]
	\item $\Ker(\varphi)$ contains no nonzero left ideals of $R$.
	\item $\varphi(rs)=\varphi(sr)$ for all $r,s\in R$.
\end{enumerate}

Following \cite[Chapter~III,~Section~9]{Bou} and \cite[Section~3]{Geu2017}, for a finitely generated projective left $R$-module $P$, we have a $\K$-linear map, which is called the \emph{trace map}:
\[\mathrm{tr}_P:
	\mathrm{End}_R(P)  \tilde{\longrightarrow} \mathrm{Hom}_R(P,R)\otimes_RP\xlongrightarrow{v} R \xlongrightarrow{\varphi} \K.\]
Here, the isomorphism is the inverse of $\mathrm{Hom}_R(P,R)\otimes_RP\tilde{\rightarrow}\mathrm{End}_R(P)$, $f\otimes x\mapsto (y\mapsto f(y)x)$ and $v(g\otimes x)=g(x)$, for any $g\in \mathrm{Hom}_R(P,R)$ and $x\in P.$ 

Let $x_1,x_2,\cdots,x_m\in P$ and $\theta_1,\theta_2,\cdots,\theta_m\in\mathrm{Hom}_R(P,R)$ be a dual basis of $P$. Then for any $f\in\mathrm{End}_R(P)$, we have
\[\mathrm{tr}_P(f)=\sum_{i=1}^m\varphi\big(\theta_i(f(x_i))\big).\]

The following lemma is well known; see \cite[Lemma~3.2.11]{Geu2017}.
\begin{lem}\label{lem3.5}
	Let $R$ be a symmetric algebra with a symmetric structure $\varphi:R\rightarrow\K$. For $M\in R$-$\mathrm{mod}$ and $P\in R$-$\proj$, we have a natural $\K$-linear isomorphism
	\begin{equation}\label{iso3.2}
		\mathfrak{t}_{M,P}:\mathrm{Hom}_R(M,P)\longrightarrow D\mathrm{Hom}_R(P,M),\ f\longmapsto (g\mapsto\mathrm{tr}_P(f\circ g)).
		\end{equation}
\end{lem}
\begin{proof}
	In fact, the map $\mathfrak{t}_{M,P}$ coincides with the composition of the following isomorphisms:
	\[\begin{aligned}
		\mathrm{Hom}_R(M,P) &\xlongrightarrow{\sim}\mathrm{Hom}_R(M,R\otimes_RP)\\ &\xlongrightarrow{\sim} \mathrm{Hom}_R(M,D(R)\otimes_RP))\\ &\xlongrightarrow{\sim} \mathrm{Hom}_R(M,D\mathrm{Hom}_R(P,R)))\\ & \xlongrightarrow{\sim}D(\mathrm{Hom}_R(P,R)\otimes_RM)\\ & 
	\xlongrightarrow{\sim}D\mathrm{Hom}_R(P,M).
	\end{aligned}\]
	Here, the first isomorphism is induced by the canonical isomorphism $P\xrightarrow{\sim} R\otimes_R P$, the second one is induced by the bimodule isomorphism $R\xrightarrow{\sim}D(R),\;r\mapsto(s\mapsto\varphi(rs))$, the third one is induced by the canonical isomorphism $D(R)\otimes_RP\xrightarrow{\sim}D\mathrm{Hom}_R(P,R), f\otimes p\mapsto(g\mapsto fg(p))$, the fourth one is the adjunction isomorphism, and the last one is induced by the isomorphism $\mathrm{Hom}_R(P,R)\otimes_RM\xrightarrow{\sim}\mathrm{Hom}_R(P,M),\;f\otimes x\mapsto (y\mapsto f(y)x)$.
\end{proof}

Recall that $A_i=\mathbb{K}X(i)$ for each $i\in Q_0$. Let $\varphi_i:A_i\rightarrow \K$, $\sum_{g\in X(i)}a_gg\mapsto a_1$. It is a symmetric structure of the group algebra $A_i$. By \cite[Lemma~2.2.25]{Geu2017}, we have an $A_i$-$A_j$-bimodule isomorphism for each arrow $\alpha:i\rightarrow j$ in $Q$:
\begin{equation}\label{equ3.2}
\mathrm{Hom}_{A_j}(\mathbb{K}X(\alpha),A_j)\longrightarrow\mathrm{Hom}_{\K}(\mathbb{K}X(\alpha),\K),\ \  f\mapsto\varphi_j\circ f.
\end{equation}
%and an $A_i$-$A_j$-bimodule isomorphism:
%\begin{equation}\label{equ3.3}
%	\mathrm{Hom}_{A_i^{op}}(\K X(\alpha),A_i)\rightarrow\mathrm{Hom}_{\K}(\K X(\alpha),\K),\ f\mapsto \varphi_i\circ f.
%\end{equation}
For an arrow $\alpha:i\rightarrow j$ in $Q$, $\mathbb{K}X(\alpha)$ is a free left $A_j$-module. Therefore, for a left $A_j$-module $N_j$, we have a functorial isomorphism
\begin{equation}\label{equ3.4}
\mathrm{Hom}_{A_j}(\mathbb{K}X(\alpha),N_j)\longrightarrow\mathrm{Hom}_{A_j}(\mathbb{K}X(\alpha),A_j)\otimes_{A_j}N_j.
\end{equation}
Combining the maps (\ref{equ3.2}), (\ref{equ3.4}) and (\ref{iso1}), we get a natural isomorphism of $A_i$-modules
\begin{equation}\label{equ3.5}\mathrm{Hom}_{A_j}(\mathbb{K}X(\alpha),N_j)\longrightarrow \mathbb{K}X(\alpha)^\ast\otimes_{A_j}N_j,\ f\mapsto\sum_{b\in L_\alpha}b^\ast\otimes f(b).\end{equation}
Here, $L_\alpha$ is a representative set of $X(t(\alpha))$-orbits of $X(\alpha)$.

Let $M_i$ be a left $A_i$-module. By the adjunction isomorphism and $(\ref{equ3.5})$, we get a functorial isomorphism of $\K$-vector spaces:
\[\mathrm{ad}_{\alpha}(M_i,N_j):\mathrm{Hom}_{A_j}(\mathbb{K}X(\alpha)\otimes_{A_i}M_i,N_j)\rightarrow\mathrm{Hom}_{A_i}(M_i,\mathbb{K}X(\alpha)^\ast\otimes_{A_j}N_j),\;f\mapsto f^\vee.\]
Here, for any $m\in M_i$, we have
\begin{equation}\label{equ3.6}
f^\vee(m)=\sum_{b\in L_\alpha}b^\ast\otimes f(b\otimes m).
\end{equation}
The inverse of $\mathrm{ad}_{\alpha}(M_i,N_j)$ is given by
\[\begin{matrix}\mathrm{Hom}_{A_i}(M_i,\mathbb{K}X(\alpha)^\ast\otimes_{A_j}N_j)&\longrightarrow&\mathrm{Hom}_{A_j}(\mathbb{K}X(\alpha)\otimes_{A_i}M_i,N_j)\\
	g&\longmapsto&\left(g^\wedge:v\otimes m\mapsto\sum_{b\in L_\alpha}b^\ast(v)g_{m,b}\right),\end{matrix}\]
where the elements $g_{m,b}\in N_j$ are uniquely determined by
\[g(m)=\sum_{b\in L_\alpha}b^\ast\otimes g_{m,b}.\] 

%For each $i\in Q_0$ and $A_i$-module $M$, we denote by $\mathfrak{t}_i(M,P)$ the isomorphism (\ref{iso3.2}) for $\Hom_{A_i}(M,P) \tilde{\rightarrow} D\Hom_{A_i}(P,M)$.

Let $M\in\mathrm{rep}(Q,X)$ and $\mathcal{C}=\mathcal{C}(Q,X)$. We have the following homomorphism of left $\mathbb{K}\mathcal{C}$-modules
\[\Phi:\bigoplus\limits_{\alpha\in Q_1}\mathrm{Hom}_{A_{s(\alpha)}}(\mathbb{K}X(\alpha)^\ast\otimes_{A_{t(\alpha)}}e_{t(\alpha)}\mathbb{K}\mathcal{C},M_{s(\alpha)})\rightarrow\bigoplus\limits_{i\in Q_0}\mathrm{Hom}_{A_i}(e_{i}\mathbb{K}\mathcal{C},M_i),\]
which sends $(\phi_\alpha)_{\alpha\in Q_1}$ to 
\[\left(x_i\mapsto\sum_{\beta\in Q_1\atop t(\beta)=i}M_{\beta}\big(\phi_\beta^\vee(x_i)\big)-\sum_{\alpha\in Q_1\atop s(\alpha)=i}\sum_{c\in L_\alpha}\phi_\alpha(c^\ast\otimes cx_i)\right)_{i\in Q_0}.\]
Here, for $x_i\in e_i\mathbb{K}\mathcal{C}$, $cx_i$ means their multiplication in $\mathbb{K}\mathcal{C}$, and we recall $M_\alpha:\mathbb{K}X(\alpha)\otimes_{A_{s(\alpha)}}M_{s(\alpha)}\rightarrow M_{t(\alpha)}$.

We have the following homomorphism of right $\mathbb{K}\mathcal{C}$-modules
\[\Psi:\bigoplus\limits_{i\in Q_0}\mathrm{Hom}_{A_i}(M_i,e_i\mathbb{K}\mathcal{C})	\rightarrow\bigoplus\limits_{\alpha\in Q_1}\mathrm{Hom}_{A_{t(\alpha)}}(\mathbb{K}X(\alpha)\otimes_{A_{s(\alpha)}}M_{s(\alpha)},e_{t(\alpha)}\mathbb{K}\mathcal{C})\]
given by
\[(\psi_i)_{i\in Q_0}\mapsto\big(x\otimes m\mapsto\psi_{t(\alpha)}(M_\alpha(x\otimes m))- x\psi_{s(\alpha)}(m)\big)_{\alpha\in Q_1}.\]
Here, $x\psi_{s(\alpha)}(m)$ means their multiplication in $\mathbb{K}\mathcal{C}$.

Recall the isomorphism $\mathfrak{t}$ in Lemma~\ref{lem3.5}. The following result is analogous to \cite[Proposition~8.3]{GLS}. The commutative diagram implies that $D(\Psi)$ can be identified with $\Phi$, which is prepared for Section~\ref{sec4}.

\begin{pro}\label{lem3.10}
	Let $M\in\mathrm{rep}_{l.p.}(Q,X)$ and $\mathcal{C}=\mathcal{C}(Q,X)$. Then we have the following commutative diagram.
	\[\begin{tikzcd}[column sep=large,row sep=huge]
		\bigoplus\limits_{\alpha\in Q_1}\mathrm{Hom}_{A_{s(\alpha)}}(\mathbb{K}X(\alpha)^\ast\otimes_{A_{t(\alpha)}}e_{t(\alpha)}\mathbb{K}\mathcal{C},M_{s(\alpha)}) \arrow[r, "\Phi"] \arrow[d, "\bigoplus\limits_{\alpha\in Q_1}\mathfrak{t}_{\mathbb{K}X(\alpha)^\ast\otimes_{A_{t(\alpha)}}e_{t(\alpha)}\mathbb{K}\mathcal{C},M_{s(\alpha)}}"swap] &\bigoplus\limits_{i\in Q_0}\mathrm{Hom}_{A_i}(e_{i}\mathbb{K}\mathcal{C},M_i) \arrow[dd, "\bigoplus\limits_{i\in Q_0}\mathfrak{t}_{e_{i}\mathbb{K}\mathcal{C},M_i}"] \\
		\bigoplus\limits_{\alpha\in Q_1}D\mathrm{Hom}_{A_{s(\alpha)}}(M_{s(\alpha)},\mathbb{K}X(\alpha)^\ast\otimes e_{t(\alpha)}\mathbb{K}\mathcal{C}) \arrow[d, "{\bigoplus\limits_{\alpha\in Q_1}D(\mathrm{ad}_\alpha(M_{s(\alpha)},e_{t(\alpha)}\mathbb{K}\mathcal{C}))}"swap]                  &                         \\
		\bigoplus\limits_{\alpha\in Q_1}D\mathrm{Hom}_{A_{t(\alpha)}}(\mathbb{K}X(\alpha)\otimes_{A_{s(\alpha)}}M_{s(\alpha)},e_{t(\alpha)}\mathbb{K}\mathcal{C}) \arrow[r, "D(\Psi)"]              & \bigoplus\limits_{i\in Q_0}D\mathrm{Hom}_{A_i}(M_i,e_i\mathbb{K}\mathcal{C})                   
	\end{tikzcd}\]
%where $\mathfrak{t}_\alpha=\mathfrak{t}_{\mathbb{K}X(\alpha)^\ast\otimes_{A_{t(\alpha)}}e_{t(\alpha)}\K\mathcal{C},M_{s(\alpha)}}$.
\end{pro}
\begin{proof}
	We write $\mathrm{tr}_{j}=\mathrm{tr}_{M_j}$ for each $j\in Q_0$. Fix an arrow $\alpha\in Q_1$ and $i\in Q_0$. For $f\in\mathrm{Hom}_{A_i}(M_i,e_i\mathbb{K}\mathcal{C})$, we have
	\[\Psi(f)(x\otimes m)=\begin{cases}
		-xf(m),&\mbox{if }s(\alpha)=i;\\
		f(M_\alpha(x\otimes m)),&\mbox{if }t(\alpha)=i;\\
		0,&\mbox{otherwise};
	\end{cases}\]
 for $x\in \mathbb{K}X(\alpha)$ and $m\in M_{s(\alpha)}$. Here, $f$ is regarded as an element of $\bigoplus_{j\in Q_0}\mathrm{Hom}_{A_j}(M_j,e_j\mathbb{K}\mathcal{C})$ with the $i$-component $f$. We have
	\[\Phi(\phi)(y)=\begin{cases}
		-\sum_{c\in L_\alpha}\phi(c^\ast\otimes cy),&\mbox{if }s(\alpha)=i;\\
		M_\alpha(\phi^\vee(y)),&\mbox{if }t(\alpha)=i;\\
		0,&\mbox{otherwise};
	\end{cases}\]
for $y\in e_i\mathbb{K}\mathcal{C}$. Here, $\phi$ is regarded as an element of $\bigoplus_{\beta\in Q_1}\mathrm{Hom}_{A_{s(\beta)}}(\mathbb{K}X(\beta)^\ast\otimes_{A_{t(\beta)}}e_{t(\beta)}\mathbb{K}\mathcal{C},M_{s(\beta)})$ with the $\alpha$-component $\phi$. 

We only need to prove
	\[\mathrm{tr}_{s(\alpha)}(\phi\circ \Psi(f)^\vee)=\mathrm{tr}_i(\Phi(\phi)\circ f).\]
Here, the right side of the equality is the composition of the upper right corner morphisms of the diagram, and the left side is the composition of the lower left corner morphisms. 

If $s(\alpha)=i$, by (\ref{equ3.6}) we have 
\[\mathrm{tr}_{s(\alpha)}(\phi\circ \Psi(f)^\vee)=-\sum_{c\in L_\alpha}\mathrm{tr}_i(\phi(c^\ast\otimes cf(-)))=\mathrm{tr}_i(\Phi(\phi)\circ f).\]

If $t(\alpha)=i$, by \cite[Lemmas~3.2.11 and ~3.2.13]{Geu2017} we have
\[\begin{aligned}\mathrm{tr}_{s(\alpha)}(\phi\circ \Psi(f)^\vee)
	&=\mathrm{tr}_i(\phi^\vee\circ \Psi(f))\\ &=\mathrm{tr}_i(\phi^\vee \circ f\circ M_\alpha)\\
	&=\mathrm{tr}_i(M_\alpha\circ \phi^\vee\circ f)\\
	&=\mathrm{tr}_i(\Phi(\phi)\circ f).
\end{aligned}\]	
The remaining case is trivial. Then we are done.
\end{proof}

\section{Bimodules and the preprojective algebra}\label{sec4}
In this section, we describe $\Hom_{\mathbb{K}\mathcal{C}}(\Pi(Q,X)_1,M)$ for any $\mathbb{K}\mathcal{C}$-module $M$ and prove Proposition~\ref{pro4.2}.

\subsection{The bimodule $\Pi(Q,X)_1$}
 
 Recall that $\Pi(Q,X)_1$ is the sub $\mathbb{K}\mathcal{C}$-$\mathbb{K}\mathcal{C}$-bimodule of $\Pi(Q,X)$ generated by $X(\alpha)^\ast$ for all $\alpha\in Q_1$. 
 %In this subsection, we prove that the $\mathbb{K}\mathcal{C}$-$\mathbb{K}\mathcal{C}$-bimodules $\Pi(Q,X)_1$ and $\mathrm{Ext}_{\mathbb{K}\mathcal{C}}(D(\mathbb{K}\mathcal{C}),\mathbb{K}\mathcal{C})$ are isomorphic.
 
Let $M\in\mathrm{rep}(Q,X)$. Recall that we have the left $\mathbb{K}\mathcal{C}$-module homomorphism\[\Phi:\bigoplus\limits_{\alpha\in Q_1}\mathrm{Hom}_{A_{s(\alpha)}}(\mathbb{K}X(\alpha)^\ast\otimes_{A_{t(\alpha)}}e_{t(\alpha)}\mathbb{K}\mathcal{C},M_{s(\alpha)})\rightarrow\bigoplus\limits_{i\in Q_0}\mathrm{Hom}_{A_i}(e_{i}\mathbb{K}\mathcal{C},M_i),\]
which sends $(\phi_\alpha)_{\alpha\in Q_1}$ to 
\[\left(x_i\mapsto\sum_{\beta\in Q_1\atop t(\beta)=i}M_{\beta}\big(\phi_\beta^\vee(x_i)\big)-\sum_{\alpha\in Q_1\atop s(\alpha)=i}\sum_{c\in L_\alpha}\phi_\alpha(c^\ast\otimes cx_i)\right)_{i\in Q_0}.\]

\begin{lem}\label{lem4.1}
	Keep the notation as above. For any $M\in\mathrm{rep}(Q,X)$, we have a left $\mathbb{K}\mathcal{C}$-module isomorphism $\mathrm{Hom}_{\mathbb{K}\mathcal{C}}(\Pi(Q,X)_1,M)\cong\mathrm{Ker}(\Phi)$.
\end{lem}
\begin{proof}
	Write $\Pi(Q,X)_1=\Pi_1.$ By Proposition~\ref{pro2.6}~(1), we have a $\mathbb{K}\mathcal{C}$-$\mathbb{K}\mathcal{C}$-bimodule isomorphism $\Pi_1\cong\mathbb{K}\mathcal{C}\otimes_A\otimes V_1\otimes_A\mathbb{K}\mathcal{C}/\mathbb{K}\mathcal{C}\rho'\mathbb{K}\mathcal{C}$. We mention that $\mathbb{K}\mathcal{C}\rho'\mathbb{K}\mathcal{C}$ is the sub $\mathbb{K}\mathcal{C}$-bimodule of $\mathbb{K}\mathcal{C}\otimes_A\otimes V_1\otimes_A\mathbb{K}\mathcal{C}$ generated by $$\sum_{\alpha\in Q_1\atop t(\alpha)=i}\sum_{b\in R_\alpha}b\otimes b^\ast\otimes e_i-\sum_{\beta\in Q_1\atop s(\beta)=i}\sum_{c\in L_\beta}e_i\otimes c^\ast\otimes c$$ for all $i\in Q_0$. Then we have a right exact sequence of $\mathbb{K}\mathcal{C}$-bimodule
	\begin{equation}\label{equ4.1}
		\mathbb{K}\mathcal{C}\otimes_A\mathbb{K}\mathcal{C}\xlongrightarrow{\partial}\mathbb{K}\mathcal{C}\otimes_A V_1\otimes_A\mathbb{K}\mathcal{C}\longrightarrow \Pi_1\longrightarrow 0,
	\end{equation}
	which $\partial$ is given by $$\partial(e_i\otimes e_i)=\sum_{\alpha\in Q_1\atop t(\alpha)=i}\sum_{b\in R_\alpha}b\otimes b^\ast\otimes e_i-\sum_{\beta\in Q_1\atop s(\beta)=i}\sum_{c\in L_\beta}e_i\otimes c^\ast\otimes c.$$
	The sequence (\ref{equ4.1}) is isomorphic to the following sequence
	\begin{equation*}
		\bigoplus_{i\in Q_0}\mathbb{K}\mathcal{C}e_i\otimes_{A_i}e_i\mathbb{K}\mathcal{C}\xlongrightarrow{\partial}\bigoplus_{\alpha\in Q_1}\mathbb{K}\mathcal{C}e_{s(\alpha)}\otimes_{A_{s(\alpha)}} \mathbb{K}X(\alpha)^\ast\otimes_{A_{t(\alpha)}}e_{t(\alpha)}\mathbb{K}\mathcal{C}\longrightarrow \Pi_1\longrightarrow 0.
	\end{equation*}
	Applying $\mathrm{Hom}_{\mathbb{K}\mathcal{C}}(-,M)$ to this sequence and using the adjunction isomorphism, we get the following commutative diagram.
	\[\begin{tikzcd}
		\bigoplus\limits_{\alpha\in Q_1}\mathrm{Hom}_{\mathbb{K}\mathcal{C}}(\mathbb{K}\mathcal{C}e_{s(\alpha)}\otimes_{A_{s(\alpha)}} \mathbb{K}X(\alpha)^\ast\otimes_{A_{t(\alpha)}}e_{t(\alpha)}\mathbb{K}\mathcal{C},M) \arrow[r, "{\partial^\ast}"] \arrow[d,"\cong"] & \bigoplus\limits_{i\in Q_0}\mathrm{Hom}_{\mathbb{K}\mathcal{C}}(\mathbb{K}\mathcal{C}e_i\otimes_{A_i}e_i\mathbb{K}\mathcal{C},M) \arrow[d,"\cong"] \\
		\bigoplus\limits_{\alpha\in Q_1}\mathrm{Hom}_{A_{s(\alpha)}}(\mathbb{K}X(\alpha)^\ast\otimes_{A_{t(\alpha)}}e_{t(\alpha)}\mathbb{K}\mathcal{C},M_{s(\alpha)}) \arrow[r, "\Phi"]                                           & \bigoplus\limits_{i\in Q_0}\mathrm{Hom}_{A_i}(e_{i}\mathbb{K}\mathcal{C},M_i)                   
	\end{tikzcd}\]
	Therefore, we have $\mathrm{Hom}_{\mathbb{K}\mathcal{C}}(\Pi_1,M)\cong\Ker(\partial^\ast)\cong\Ker(\Phi)$.
\end{proof}
\subsection{Auslander-Reiten translations}

Let $A$ be a finite dimensional $\mathbb{K}$-algebra, and let $\mathcal{P}^{\le 1}(A)$ be the full subcategory of $A$-$\mathrm{mod}$ consisting of all $A$-modules $M$ with $\projdim(M)\le 1$. Then the Auslander-Reiten translation $\tau:\mathcal{P}^{\le 1}\rightarrow A$-$\mathrm{mod}$ is a functor; see \cite[Proposition~IV.1.13]{ARS}. Similarly, let $\mathcal{I}^{\le 1}(A)$ be the full subcategory of $A$-$\mathrm{mod}$ consisting of all $A$-modules $N$ with $\injdim(N)\le 1$. Then $\tau^-:\mathcal{I}^{\le 1}(A)\rightarrow A$-$\mathrm{mod}$ is a functor. By Proposition~\ref{pro3.3} , we have functors $\tau, \tau^-:\mathrm{rep}_{l.p.}(Q,X)\rightarrow\mathrm{rep}(Q,X)$.

In this subsection, we prove the following result; compare \cite[Theorem~10.1]{GLS}.

\begin{pro}\label{pro4.2}
	Let $(Q,X)$ be a finite acyclic EI quiver with an action-free assignment $X$. For $M\in\mathrm{rep}_{l.p.}(Q,X)$, we have functorial isomorphisms
	\[\mathrm{Hom}_{\mathbb{K}\mathcal{C}}(\Pi(Q,X)_1,M)\cong \tau(M)\mbox{ and }\Pi(Q,X)_1\otimes_{\mathbb{K}\mathcal{C}} M\cong\tau^-(M).\]
\end{pro}
\begin{proof}
	For $M\in\mathrm{rep}_{l.p.}(Q,X)$, by Lemma \ref{lem3.2}, we have a projective resolution of $M$,
	\begin{equation}\label{equ4.2.0}0 \to \bigoplus\limits_{\alpha\in Q_1}\mathbb{K}\mathcal{C} e_{t(\alpha)}\otimes_{A_{t(\alpha)}}\mathbb{K}X(\alpha)\otimes_{A_{s(\alpha)}}M_{s(\alpha)} \xrightarrow{d} \bigoplus\limits_{i\in Q_0}\mathbb{K}\mathcal{C} e_i\otimes_{A_i}M_i \xrightarrow{\mu} M \to 0,\end{equation}
	where $d(a\otimes v\otimes m)=a\otimes M_\alpha(v\otimes m)-av\otimes m,a\in \mathbb{K}\mathcal{C} e_{t(\alpha)},v\in \mathbb{K}X(\alpha),m\in M_{s(\alpha)}$. We have $\Ker(\nu(d))\cong\tau(M)$, where $\nu=D\mathrm{Hom}_{\mathbb{K}\mathcal{C}}(-,\mathbb{K}\mathcal{C})$ is the Nakayama functor.
	
%	Applying Nakayama functor $\nu=D\mathrm{Hom}_{\mathbb{K}\mathcal{C}}(-,\K\mathcal{C})$ to (\ref{equ4.2.0}), we have an exact sequence
%	\[\begin{aligned}
%		0 \to \tau(TM) &\to \bigoplus_{\alpha\in Q_1}D\mathrm{Hom}_{\K\mathcal{C}}(\K\mathcal{C} e_{t(\alpha)}\otimes_{A_{t(\alpha)}}\K X(\alpha)\otimes_{A_{s(\alpha)}}M_{s(\alpha)},\K\mathcal{C}) \\ &\xrightarrow{\nu(d)} \bigoplus_{i\in Q_0}D\mathrm{Hom}_{\K\mathcal{C}}(\K\mathcal{C} e_i\otimes_{A_i}M_i,\K\mathcal{C}).\end{aligned}\]
	By Proposition~\ref{lem3.10}, we have commutative diagram
	\[\begin{tikzcd}[column sep=small]
		\bigoplus\limits_{\alpha\in Q_1}\mathrm{Hom}_{A_{s(\alpha)}}(\mathbb{K}X(\alpha)^\ast\otimes_{A_{t(\alpha)}}e_{t(\alpha)}\mathbb{K}\mathcal{C},M_{s(\alpha)}) \arrow[r, "\Phi"] \arrow[d, "\cong"]    & \bigoplus\limits_{i\in Q_0}\mathrm{Hom}_{A_i}(e_{i}\mathbb{K}\mathcal{C},M_i) \arrow[dd, "\cong"] \\
		\bigoplus\limits_{\alpha\in Q_1}D\Hom_{A_{s(\alpha)}}(M_{s(\alpha)},\mathbb{K}X(\alpha)^\ast\otimes_{A_{t(\alpha)}}e_{t(\alpha)}\mathbb{K}\mathcal{C}) \arrow[d, "\cong"]                      &                       \\
		\bigoplus\limits_{i\in Q_0}D\Hom_{A_{s(\alpha)}}(\mathbb{K}X(\alpha)^\ast\otimes_{A_{s(\alpha)}}M_{s(\alpha)},e_{t(\alpha)}\mathbb{K}\mathcal{C}) \arrow[r, "D(\Psi)"]                  & \bigoplus\limits_{i\in Q_0}D\mathrm{Hom}_{A_i}(M_i,e_i\mathbb{K}\mathcal{C})                    \\
		\bigoplus\limits_{\alpha\in Q_1}D\mathrm{Hom}_{\mathbb{K}\mathcal{C}}(\mathbb{K}\mathcal{C} e_{t(\alpha)}\otimes_{A_{t(\alpha)}}\mathbb{K}X(\alpha)^\ast\otimes_{A_{s(\alpha)}}M_{s(\alpha)},\mathbb{K}\mathcal{C}) \arrow[u, "\cong"'] \arrow[r, "\nu(d)"] &  \bigoplus\limits_{i\in Q_0}D\mathrm{Hom}_{\mathbb{K}\mathcal{C}}(\mathbb{K}\mathcal{C} e_i\otimes_{A_i}M_i,\mathbb{K}\mathcal{C})         \arrow[u, "\cong"']
	\end{tikzcd}\]
	Here, the two isomorphisms at the bottom are induced by the adjunction isomorphisms. Thus $\tau(M)\cong\Ker(Phi)$. We write $\Pi_1=\Pi(Q,X)_1$ By Lemma~\ref{lem4.1}, we have \begin{equation}\label{equ4.6}\tau(M)\cong\Ker(\Phi)\cong\Hom_{\mathbb{K}\mathcal{C}}(\Pi_1,M).\end{equation}
	
	Let $M\in\mathrm{rep}_{l.p.}(Q,X)$. By Proposition~\ref{pro3.3}, we have $\injdim(M)\le 1$ and $\projdim(D(\mathbb{K}\mathcal{C}))\le 1$. Therefore, we have
	\begin{equation*}\label{equ4.4}
		\begin{aligned}\Hom_{\mathbb{K}\mathcal{C}}(\tau^-(M),D(\mathbb{K}\mathcal{C}))&\cong\Hom_{\mathbb{K}\mathcal{C}}(M,\tau D(\mathbb{K}\mathcal{C}))\\ &\cong\Hom_{\mathbb{K}\mathcal{C}}(M,\Hom_{\mathbb{K}\mathcal{C}}(\Pi_1,D(\mathbb{K}\mathcal{C})))\\ &\cong\Hom_{\mathbb{K}\mathcal{C}}(\Pi_1\otimes_{\mathbb{K}\mathcal{C}} M,D(\mathbb{K}\mathcal{C})).
		\end{aligned}
		\end{equation*}
	Here, the first isomorphism is \cite[Corollary~IV~2.15]{ASS}, the second one is induced by (\ref{equ4.6}) and the last one is the adjunction isomorphism. Since $\Hom_{\mathbb{K}\mathcal{C}}(N,D(\mathbb{K}\mathcal{C}))\cong D(\mathbb{K}\mathcal{C}\otimes_{\mathbb{K}\mathcal{C}}N)\cong D(N)$ for the left $\mathbb{K}\mathcal{C}$-module $N$, we get $\tau^-(M)\cong \Pi_1\otimes_{\mathbb{K}\mathcal{C}} M$. Then we are done. 
\end{proof}

\section{The main result}\label{sec5}  In this section, we prove the main result of this paper; see Theorem~\ref{thm4.3}. Keep the notation of Section~\ref{sec4}.

\begin{pro}\label{pro4.3}
	Let $(Q,X)$ be a finite acyclic EI quiver with an action-free assignment $X$ and let $\mathcal{C}=\mathcal{C}(Q,X)$. Then we have a $\mathbb{K}\mathcal{C}$-$\mathbb{K}\mathcal{C}$-bimodule isomorphism $\Pi(Q,X)_1\cong\Ext_{\mathbb{K}\mathcal{C}}^1(D(\mathbb{K}\mathcal{C}),\mathbb{K}\mathcal{C})$. 
\end{pro}
\begin{proof}
	We write $\Lambda=\mathbb{K}\mathcal{C}$ and $\Pi_1=\Pi(Q,X)_1$. The proof method is the same as in \cite[Theorem 10.5]{GLS}. By the proposition  \ref{pro4.2}, for $M\in\mathrm{rep}_{l.p.}(Q,X)$, we have a functorial isomorphism
	\begin{equation}\label{equ4.5}
		\Pi_1\otimes_\Lambda M\cong\tau^-(M)\cong\Ext_\Lambda^1(D(\Lambda),M).
	\end{equation}
	Since $\projdim(D(\Lambda))\le 1$, we have a right exact functor
	\[\mathrm{Ext}_\Lambda^1(D(\Lambda),-):\mathrm{rep}(Q,X)\rightarrow\mathrm{rep}(Q,X).\]
	By \cite[Corollary~5.47]{Rot}, we have $\Ext_\Lambda^1(D(\Lambda),-)\cong\Ext_\Lambda^1(D(\Lambda),\Lambda)\otimes_\Lambda-$. 
	
	For $N\in\mathrm{rep}(Q,X)$, let
	\begin{equation}\label{equ4.3}P_1\rightarrow P_0\rightarrow N\rightarrow 0\end{equation}
	be a projective presentation of $N$. Applying the right exact functors $\Pi_1\otimes_\Lambda-$ and $\Ext_\Lambda^1(D(\Lambda),-)$ to (\ref{equ4.3}) yields a functorial commutative diagram
	\[\begin{tikzcd}
		\Pi_1\otimes_\Lambda P_1 \arrow[r] \arrow[d, "\eta_{P_1}"] & \Pi_1\otimes_\Lambda P_0 \arrow[r] \arrow[d, "\eta_{P_0}"] & \Pi_1\otimes_\Lambda N\arrow[r] \arrow[d, "\eta_N"] & 0 \\
		\Ext_\Lambda^1(D(\Lambda),P_1) \arrow[r]                         & \Ext_\Lambda^1(D(\Lambda),P_0) \arrow[r]                         & \Ext_\Lambda^1(D(\Lambda),N) \arrow[r]                     & 0
	\end{tikzcd}\]
	with exact rows. Since the functorial isomorphism (\ref{equ4.5}), we obtain that $\eta_{P_0}$ and $\eta_{P_1}$ are isomorphisms. Then $\eta_N$ is also an isomorphism. It follows that the functor $\Pi_1\otimes_\Lambda-$ is naturally isomorphic to $\Ext_\Lambda(D(\Lambda),-)$. Then $\Pi_1\otimes_\Lambda-$ is naturally isomorphic to $\Ext_\Lambda^1(D(\Lambda),\Lambda)\otimes_\Lambda-$.  Therefore we have $\Pi_1\cong\Ext_\Lambda(D(\Lambda),\Lambda)$ as $\Lambda$-$\Lambda$-bimodules.
\end{proof}

\begin{thm}\label{thm4.3}
	Let $\mathcal{C}=\mathcal{C}(Q,X)$ be the category associated to a finite acyclic EI quiver $(Q,X)$ such that $X$ is action-free. Then we have an algebra isomorphism $\Pi(Q,X)\cong T_{\mathbb{K}\mathcal{C}}(\Ext_{\mathbb{K}\mathcal{C}}^1(D(\mathbb{K}\mathcal{C}),\mathbb{K}\mathcal{C}))$.
\end{thm}
\begin{proof}
	By Proposition~\ref{pro2.6}, we have an algebra isomorphism $$\Pi(Q,X)\cong T_{\mathbb{K}\mathcal{C}}(\Pi(Q,X)_1).$$ By Proposition~\ref{pro4.3}, we get a $\mathbb{K}\mathcal{C}$-$\mathbb{K}\mathcal{C}$-bimodule isomorphism $$\Pi(Q,X)_1\cong\Ext_{\mathbb{K}\mathcal{C}}^1(D(\mathbb{K}\mathcal{C}),\mathbb{K}\mathcal{C}).$$ Therefore, we have an algebra isomorphism 
	$$
	T_{\mathbb{K}\mathcal{C}}(\Pi(Q,X)_1)\cong T_{\mathbb{K}\mathcal{C}}(\Ext_{\mathbb{K}\mathcal{C}}^1(D(\mathbb{K}\mathcal{C}),\mathbb{K}\mathcal{C})).
	$$
	Then we have the isomorphism of algebras $\Pi(Q,X)\cong T_{\mathbb{K}\mathcal{C}}(\Ext_{\mathbb{K}\mathcal{C}}^1(D(\mathbb{K}\mathcal{C}),\mathbb{K}\mathcal{C}))$.
\end{proof}

\section{The algebras associated to Cartan matrices}\label{sec6}
In this section, we recall three algebras, which are associated to symmetrizable generalized Cartan matrices; see \cite{GLS,CW}. For the EI quiver of Cartan type in \cite{CW}, we prove that the corresponding preprojective algebra is isomorphic to the generalized preprojective algebra in \cite{GLS}; see Theorem~\ref{pro6.2}.

Let $n\ge 1$ be a positive integer. A matrix $C=(c_{ij})\in M_n(\mathbb{Z})$ is called a \emph{symmetrizable generalized Cartan matrix}, if the following conditions hold:
\begin{enumerate}
	\item[(C1)] $c_{ii}=2$ for all $i$;
	\item[(C2)] $c_{ij}\leq 0$ for all $i\neq j$, and $c_{ij}<0$ if and only if $c_{ji}<0$;
	\item[(C3)] There is a diagonal  matrix $D={\rm diag}(c_1,\cdots,c_n)$  with $c_i\in  \mathbb{Z}_{\geq 1}$ for all $i$ such that the product matrix $DC$ is symmetric.
\end{enumerate}
The matrix $D$ appearing in (C3) is called a \emph{symmetrizer} of $C$. To be concise, a symmetrizable generalized Cartan matrix will be called a Cartan matrix. For all $c_{ij}<0$, let
\[g_{ij}:=|\gcd(c_{ij},c_{ji})|,\mbox{ and } f_{ij}:=\frac{|c_{ij}|}{g_{ij}}.\]

Let $C=(c_{ij})$ be a Cartan matrix. An \emph{orientation} of $C$ is a subset $\Omega\subset \{1,2,\cdots,n\}\times \{1,2,\cdots,n\}$ such that the following conditions are satisfied:
\begin{enumerate}
	\item[(O1)] $\{(i,j),(j,i)\}\cap \Omega\neq \emptyset$ if and only if $c_{ij}<0$;
	\item[(O2)] For each sequence $(i_1,i_2,\cdots,i_t,i_{t+1})$ with $t\geq 1$ and $(i_s,i_{s+1})\in \Omega$ for all $1\leq s\leq t$,  we have $i_1\neq i_{t+1}$.
\end{enumerate}

Let $(C,D,\Omega)$ be a \emph{Cartan triple}, that is, $C$ is a Cartan matrix, $D$ its symmetrizer and $\Omega$ an orientation of $C$. In what follows, we recall a finite free EI category $\mathcal{C}(C, D, \Omega)$ \cite{CW}, a finite dimensional algebra $H(C, D, \Omega)$ and a generalized preprojective algebra $\Pi(C,D,\Omega)$ \cite{GLS}.

Let $Q=Q(C,\Omega)$ be the finite quiver with the set of vertices $Q_0=\{1,2, \cdots,n\}$ and with the set of arrows
\[Q_1=\{\alpha^{(g)}_{ij}\colon j\rightarrow i\mid(i,j)\in \Omega, 1\leq g\leq g_{ij}\}\cup\{\varepsilon_i\colon i\rightarrow i\mid 1\leq i\leq n\}.\]

Let $Q^{\circ}=Q^{\circ}(C,\Omega)$ be the quiver obtained from $Q$ by deleting all the loops $\varepsilon_i$. By the condition (O2), the finite quiver $Q^{\circ}$ is acyclic. The finite EI quiver $(Q^\circ, X)$ is defined in \cite[Subsection~4.1]{CW}.  The assignment $X$ is given as follows: $X(i)=\langle \eta_i\; |\; \eta_i^{c_i}=1\rangle$ is a cyclic group of order $c_i$; for each $(i, j)\in \Omega$, let $G_{ij}=\langle \eta_{ij}\; |\; \eta_{ij}^{g_{ij}}=1\rangle$ be a cyclic group of order $g_{ij}$. We have the following injective group homomorphisms
\[G_{ij}\hookrightarrow X(i), \  \eta_{ij}\mapsto \eta_i^{\frac{c_i}{g_{ij}}},\]
and
\[G_{ij}\hookrightarrow X(j), \  \eta_{ij}\mapsto \eta_j^{\frac{c_j}{g_{ij}}}.\]
Then $X(i)\times_{G_{ij}} X(j)$ is a $(X(i), X(j))$-biset. We set
\begin{equation}\label{equ6.0}X(\alpha_{ij}^{(g)})=X(i)\times_{G_{ij}} X(j),\end{equation}
for each $1\leq g\leq g_{ij}$. We obtain a finite EI quiver $(Q^\circ,X)$. Denote by $\mathcal{C}(C,D,\Omega)$ the free EI category $\mathcal{C}(Q^\circ,X)$.

The algebra $H(C,D,\Omega)$ is defined as $\mathbb{K}Q/I$ \cite{GLS}, where $\mathbb{K}Q$ is the path algebra of $Q=Q(C,\Omega)$, and $I$ is the two-sided ideal of $\mathbb{K}Q$ defined by  the following relations 
$$\varepsilon_k^{c_k}=0, \; \varepsilon_i^{f_{ji}} \alpha^{(g)}_{ij}=\alpha^{(g)}_{ij}\varepsilon_j^{f_{ij}},
$$
for all $k\in Q_0,\;(i,j)\in \Omega,\;1\leq g\leq g_{ij}.$

%The relationship between two algebras $\mathbb{K}\mathcal{C}(C,D,\Omega)$ and $H(C,D,\Omega)$ is studied in \cite{CW}.

Let $(C,D,\Omega)$ be a Cartan triple and $Q=Q(C,\Omega)$. The \emph{oppsite orientation} of $\Omega$ is $\Omega^\ast=\{(j,i)|(i,j)\in\Omega\}$. Let $\overline{\Omega}=\Omega\sqcup\Omega^\ast$ and 
\[\overline{\Omega}(-,i)=\{i\in Q_0\;|\;(j,i)\in\overline{\Omega}\},\;\sgn(i,j)=\begin{cases}
	1,&\mbox{if }(i,j)\in\Omega;\\
	-1,&\mbox{if }(i,j)\in \Omega^\ast.
\end{cases}\]
The quiver $\widetilde{Q}=\widetilde{Q}(C,\Omega)$ is obtained from $Q$ by adding a new arrows $\alpha_{ji}^{(g)}:i\rightarrow j$ for each arrow $\alpha_{ij}^{(g)}:j\rightarrow i$ of $Q^\circ$.

\begin{dfn}{ (\cite[Section~1.4]{GLS})}
	The \emph{preprojective algebra} $\Pi(C,D,\Omega)$ of type $C$ is defined as the $\mathbb{K}$-algebra $\mathbb{K}\widetilde{Q}/\widetilde{I}$,  
	where $\widetilde{I}$ is the two-sided ideal of the path algebra $\mathbb{K}\widetilde{Q}$ defined by the following relations:
	\[\varepsilon_k^{c_k}=0,\;\varepsilon_i^{f_{ji}}\alpha_{ij}^{(g)}=\alpha_{ij}^{(g)}\varepsilon_j^{f_{ij}},\]
	for all $k\in Q_0$, $(i,j)\in\overline{\Omega}$ and $1\le g\le g_{ij}$; 
	\[\sum\limits_{j\in\overline{\Omega}(-,i)}\sum\limits_{g=1}^{g_{ji}}\sum\limits_{f=0}^{f_{ji}-1}\sgn(i,j)\varepsilon_i^f\alpha_{ij}^{(g)}\alpha_{ji}^{(g)}\varepsilon_i^{f_{ji}-1-f}=0,\]
	for all $i\in Q_0$.
\end{dfn}

Following \cite[Subsection~4.2]{CW}, there is a construction of a new Cartan triple $(C',D',\Omega')$ from a given Cartan triple $(C,D,\Omega)$ as follows. We recall that $C=(c_{ij})\in M_n(\mathbb{Z})$ and $D={\rm diag}(c_1, \cdots, c_n)$.

\vskip 5pt

{\bf{Case 1.}} Assume that ${\rm char}(\K)=p>0$. Let $c_i=p^{r_i}d_i$ with $r_i\geq 0$ and ${\rm gcd}(p, d_i)=1$. For each $1\leq i, j\leq n$, let
\[\Sigma_{ij}=\{ (l_i,l_j)\; |\;  0\leq l_i< d_i, 0\leq l_j< d_j, l_ip^{r_i}\equiv l_jp^{r_j} ({\rm mod}\ {\rm gcd}(d_i, d_j))\}.\]
Let
$$M=\bigsqcup_{1\leq i\leq n} \{(i, l_i)\; |\; 0\leq l_i<d_i\}.$$ 
The Cartan matrix $C'$, whose rows and columns are indexed by the set $M$, is given by:
\[c'_{(i, l_i), (j, l_j)}=\begin{cases}
	2,&\mbox{if }(i,l_i)=(j,l_j);\\
	-{\rm gcd}(c_{ij},c_{ji}) p^{r_j-{\rm min}(r_i, r_j)}, & \mbox{if }(i,l_i)\neq (j,l_j)\mbox{ and } (l_i, l_j)\in \Sigma_{ij}; \\
	0, & \mbox{otherwise.}
\end{cases}\]
Let $D'$  be a diagonal matrix, whose $(i, l_i)$-th component is given by $p^{r_i}$.  Set
$$\Omega'=\{((i, l_i),(j, l_j))\; |\; (i, j)\in \Omega, (l_i, l_j)\in \Sigma_{ij}\},$$
Then $(C',D',\Omega')$ is a Cartan triple.

{\bf{Case 2.}} If $\mathrm{char}(\K)=0$, we put $r_i=0$ and $d_i=c_i$. 
 We define the index set $M$ and the matrix $C'$ as above, which is symmetric. The matrix $D'$ is the identity matrix.

We say that $\K$ has enough roots of unity for $D$, if for each $1\leq i\leq n$, the polynomial $t^{c_i}-1$ splits in $\K[t]$. Based on \cite{CW}, we have the following result.
\begin{thm}\label{pro6.2}
	Assume that $(C, D, \Omega)$ is a Cartan triple and that the field $\mathbb{K}$ has enough roots of unity for $D$. Keep the notation in the above construction, and let $(Q^\circ,X)$ be the EI quiver associated to the Cartan triple $(C,D,\Omega)$.  Then there is an isomorphism of algebras
	\[\Pi(Q^\circ,X)\cong\Pi(C',D',\Omega').\]
\end{thm}
\begin{proof}
	We write $H=H(C',D',\Omega')$ and $\Lambda=\mathbb{K}\mathcal{C}(C,D,\Omega)$. By \cite[Theorem~4.3]{CW}, we have an isomorphism of algebras $\Theta:\Lambda\rightarrow H$. Then $H$ is a $\Lambda$-$\Lambda$-bimodule which is induced by $\Theta$ and $\Theta$ becomes an isomorphism of $\Lambda$-$\Lambda$-bimodules. Therefore, we have an isomorphism of $\Lambda$-$\Lambda$-bimodules
	\begin{equation*}\Ext_{\Lambda}^1(D(\Lambda),\Lambda)\cong\Ext_\Lambda^1(D(H),H)\cong\Ext_H^1(D(H),H).\end{equation*}
	Here, the isomorphism on the right side uses the isomorphism between $\Lambda$-$\mathrm{mod}$ and $H$-$\mathrm{mod}$, and the $\Lambda$-$\Lambda$-bimodule structure on $\Ext_H^1(D(H),H)$ is induced by the algebra isomorphism $\Theta$. Then we have an algebra isomorphism 
	\begin{equation}\label{equ6.1}
	T_\Lambda(\Ext_{\Lambda}^1(D(\Lambda),\Lambda))\cong T_H(\Ext_H^1(D(H),H)).
	\end{equation}
	 By \cite[Theorem~1.6]{GLS}, we have an isomorphism of algebras
	\begin{equation}\label{equ6.2} \Pi(C',D',\Omega')\cong T_H(\Ext_H^1(D(H),H)).\end{equation}
In view of (\ref{equ6.0}), we observe that the assignment $X$ is action-free. By Theorem~\ref{thm4.3}, we get 
	\begin{equation}\label{equ6.3}\Pi(Q^\circ,X)\cong 	T_\Lambda(\Ext_{\Lambda}^1(D(\Lambda),\Lambda).\end{equation} Then combining  (\ref{equ6.1}), (\ref{equ6.2}) and (\ref{equ6.3}), we obtain the required isomorphism.
\end{proof}
\begin{cor}\label{cor6.3}
	Assume that the field $\K$ has characteristic $p>0$ and that $(C,D,\Omega)$ is a Cartan triple such that each $c_i$ is a $p$-power. Let $(Q^\circ,X)$ be the EI quiver associated to the Cartan triple $(C,D,\Omega)$. Then there is a $\K$-algebra isomorphism 
	\[\Pi(Q^\circ,X)\cong\Pi(C,D,\Omega).\]
\end{cor}
\begin{proof}
	Since each $c_i$ is a $p$-power, we have $(C',D',\Omega')=(C,D,\Omega)$. Here, we identify $(i,0)$ with $i$. For the same reason, the field $\K$ has enough roots of unity for $D$. Then the required isomorphism follows from Theorem~\ref{pro6.2}.
\end{proof}
\begin{exm}
	Let $\mathrm{char}(\K)=2,C=\begin{pmatrix}
		2&-1\\-2&2
	\end{pmatrix}$ with symmetrizer $D=\mathrm{diag}\{2,1\}$ and $\Omega=\{(1,2)\}$. The finite EI quiver $(Q^\circ ,X)$ is given as follows:
	\[Q^\circ :\begin{tikzcd}
		1 & 2 \arrow[l, "\alpha_{12}"']
	\end{tikzcd},\;X(1)=\left<\eta|\eta^2=1\right>,\;X(2)=\{1\},\;X(\alpha_{12})=\{(1,1), (\eta,1)\}.\]
	Let $\mathcal{C}=\mathcal{C}(Q^\circ,X)$ and $\overline{\mathcal{C}}=\mathcal{C}(\overline{Q^\circ},\overline{X})$.
	The preprojective algebra of the finite EI quiver $(Q^\circ,X)$ is $$\Pi(Q^\circ, X)=\mathbb{K}\overline{\mathcal{C}}/\left<(1,1)(1,1)^\ast+(\eta,1)(\eta,1)^\ast-(1,1)^\ast(1,1)\right>.$$
	
	We have $(C',D',\Omega')=(C,D,\Omega)$. The preprojective algebra $\Pi(C,D,\Omega)$ is given by the quiver
	\[\begin{tikzcd}
		1 \arrow[r, "\alpha_{21}", shift left] \arrow["\varepsilon_1"', loop, distance=2em, in=125, out=55] & 2 \arrow[l, "\alpha_{12}", shift left] \arrow["\varepsilon_2", loop, distance=2em, in=55, out=125]
	\end{tikzcd}\]
	with relations $\varepsilon_1^2=0,\;\varepsilon_2=0,\;\alpha_{12}\alpha_{21}\varepsilon_1+\varepsilon_1\alpha_{12}\alpha_{21}=0$ and $-\alpha_{21}\alpha_{12}=0$.

	By Corollary~\ref{cor6.3}, we have an isomorphism of algebras $\Pi(Q^\circ,X)\cong\Pi(C,D,\Omega)$.
\end{exm}

\vskip 10pt

\noindent {\bf Acknowledgements.}\quad The author is grateful to his supervisor Professor Xiao-Wu Chen for his guidance.
%\bibliographystyle{plain}
%\bibliography{references.bib}

\frenchspacing
\bibliography{}

\vskip 10pt
{\footnotesize \noindent Dongdong Hu\\
	School of Mathematical Sciences, University of Science and Technology of China, Hefei 230026, Anhui, PR China}

%\vskip 10pt

 %{\footnotesize \noindent Xiao-Wu Chen, Ren Wang\\
 %Key Laboratory of Wu Wen-Tsun Mathematics, Chinese Academy of Sciences,\\
 %School of Mathematical Sciences, University of Science and Technology of China, Hefei 230026, Anhui, PR China}

\end{document}